\documentclass{amsart}
\usepackage{amssymb, amsmath}
\newtheorem{thm}{Theorem}[section]
\newtheorem{prop}[thm]{Proposition}
\newtheorem{lemma}[thm]{Lemma}

\newtheorem{defn}[thm]{Definition}
\newtheorem{example}[thm]{Example}
\newtheorem{remark}[thm]{Remark}

\numberwithin{equation}{section}

\begin{document}{\allowdisplaybreaks[4]

\title[Hard Lefschetz actions in  Riemannian geometry]{Hard Lefschetz actions  in   Riemannian  geometry with special holonomy}

\author{Naichung Conan Leung}
\address{The Institute of Mathematical Sciences and Department of Mathematics\\
           The Chinese University of Hong Kong\\ Shatin, Hong Kong}
\email{leung@ims.cuhk.edu.hk}
\thanks{  }

\author{Changzheng Li}
\address{The Institute of Mathematical Sciences and Department of Mathematics\\           The Chinese University of Hong Kong\\
                  Shatin, Hong Kong} \email{czli@math.cuhk.edu.hk}


\date{
      }




\begin{abstract}
     It is known that the
     hard Lefschetz action, together with
     K\"ahler identities for K\"ahler (resp. hyperk\"ahler)
     manifolds, determines a $\mathfrak{su}(1,1)_{sup}$
     (resp. $\mathfrak{sp}(1,1)_{sup}$) Lie superalgebra action on
     differential forms. In this paper, we explain the geometric
     origin of this action, and we also generalize it to manifolds
     with other holonomy groups.

       For semi-flat Calabi-Yau (resp. hyperk\"ahler) manifolds,
       these symmetries can be enlarged to a
       $\mathfrak{so}(2,2)_{sup}$ (resp.
       $\mathfrak{su}(2,2)_{sup}$) action.

\end{abstract}

\maketitle


\section{Introduction}

  Lefschetz's work (see e.g. \cite{af}) related the topology of a complex
  projective manifold $M$ with its hyperplane section. In modern
  terminology, this implies the cohomology group of $M$ admits a
  natural $\mathfrak{sl}(2,\mathbb{R})$ action. This is the
  celebrated hard Lefschetz theorem. Hodge (see e.g. \cite{gri})
  reinterpreted this action on the level of differential forms
  $\Omega^\bullet(M)$ which commutes with Laplacian operator. Thus
  the hard Lefschetz theorem follows from the Hodge theorem.
  Furthermore if we consider the vector space $\mathbb{C}^2\oplus
  \mathbb{R}$ spanned by $\partial, \bar\partial, \partial^*,
  \bar\partial^*$ and $\Delta$, then all K\"ahler identities, for
  instances $[L, \partial^*]=i\bar\partial$ and
  $\Delta=2\Delta_{\bar\partial}$, can be combined with the hard
  Lefschetz action to give a Lie superalgebra action of
  $\mathfrak{sl}(2,\mathbb{R})\oplus \mathbb{C}^2\oplus \mathbb{R}$
  on $\Omega^\bullet(M)$.

     There is an analogous   theorem for hyperk\"ahler
     manifolds $M$, namely there is a Lie superalgebra action of
     $\mathfrak{so}(4,1)\oplus\mathbb{C}^4\oplus\mathbb{R}$ on $\Omega^\bullet(M)$.
     The $\mathfrak{so}(4,1)$ part of this action  on $H^*(M)$, by zeroth order
     operators,  was discovered by Verbitsky in \cite{verb}. Following a
     suggestion of Witten, Figueroa-O'Farrill, K\"{o}hl
    and Spence \cite{fks} gave a physical interpretation of all these
    actions in terms of supersymmetric algebra in sigma models. It was
    further studied by Cao and Zhou in \cite{caozhou}.

  The followings are two natural questions which will be answered in
  this paper: (1) What is the geometric origin of these Lie
  superalgebra actions on the spaces of differential forms on
  K\"ahler manifolds (i.e. $U(n)$ holonomy) and hyperk\"ahler
  manifolds (i.e. $Sp(n)$ holonomy)? (2) Are there analogous hard
  Lefschetz type results for manifolds with other holonomy groups,
  for example quaternionic-K\"ahler manifolds, $G_2$-manifolds and
  $Spin(7)$-manifolds?

   In \cite{leung} the first author revisited the Berger
   classification of holonomy groups of Riemannian manifolds which
   are not locally symmetric spaces. Given any normed
   algebra $\mathbb{K}$, which must be one of $\mathbb{R},
   \mathbb{C}, \mathbb{H}$ and $\mathbb{O}$, we defined the notion
   of $\mathbb{K}$-manifolds. Their holonomy groups are precisely
   $O(n), U(n), Sp(n)Sp(1)$ and $Spin(7)$ respectively. If they are also
   $\mathbb{K}$-oriented, then their holonomy groups reduce to
   $SO(n), SU(n), Sp(n)$ and $G_2$ respectively.

    Note that $\mathfrak{sl}(2, \mathbb{R})\cong \mathfrak{su}(1,1)$
    and $\mathfrak{so}(4, 1)\cong \mathfrak{sp}(1,1)$. For any
    normed algebra $\mathbb{K}$, we could define analogously a Lie
    algebra $\mathfrak{su}_\mathbb{K}(1,1)$ and  a Lie superalgebra
    $\mathfrak{su}_\mathbb{K}(1,1)_{sup}=
  \mathfrak{su}_\mathbb{K}(1,1)\oplus\mathbb{K}^{1,1}\oplus
  \mathbb{R}$.
   On any  $\mathbb{K}$-manifold $M$, we will construct a natural Lie
 superalgebra bundle $E^{\mathfrak{su}}$ with fiber
 $\mathfrak{su}_\mathbb{K}(1,1)_{sup}$.
   To relate this to the hard Lefschetz action, we use the fact that
   differential forms on $M$ can be regarded as spinors for the
   direct sum $T\oplus T^*$ of the tangent and cotangent bundles of
   $M$, which admits a tautological quadratic form of type $(m, m)$.
   Roughly speaking, we have the following bundle,
           $$\mathbb{K}^{n,n}\rightarrow T\oplus T^*\rightarrow M.$$
   Using the Clifford algebra for $\mathbb{K}^{n, n}$ and the Dirac
   operator, we  construct differential operators of
   order zero, one and two on $\Omega^\bullet(M)$.  For example,
   the second order operator is simply the Laplacian operator $\Delta$.  We will show that all these operators
   together with their commutating relations, which in case of K\"ahler manifolds are   the hard Lefschetz action and  K\"ahler identities,
   generate a Lie superalgebra  $\mathfrak{su}_\mathbb{K}(1,1)_{sup}$ action.
   We have

   \begin{thm}
     Let $M$ be an oriented Riemannian manifold. Suppose $M$ is a  $\mathbb{K}$-manifold
     with $\mathbb{K}$ a normed  algebra, i.e. $\mathbb{K}\in\{\mathbb{R}, \mathbb{C}, \mathbb{H}, \mathbb{O}\}$.
     Then there is a Lie superalgebra bundle $E^{\mathfrak{su}}$
     over $M$ with fiber $\mathfrak{su}_\mathbb{K}(1,1)_{sup}$:
            $$\mathfrak{su}_\mathbb{K}(1,1)\oplus
            \mathbb{K}^{1,1}\oplus\mathbb{R}\rightarrow
            E^{\mathfrak{su}}\rightarrow M.$$

      When $\mathbb{K}$ is associative, i.e. $\mathbb{K}\neq
     \mathbb{O}$, each section of  $E^{\mathfrak{su}}\rightarrow M$
     determines a differential operator   of order at most two on
     differential forms on $M$. Thus, we have
        $$\Psi: \Gamma(M, E^{\mathfrak{su}})\rightarrow
          \mathrm{Diff}({\textstyle\bigwedge^\bullet T^*, \bigwedge^\bullet T^*}).$$

 Furthermore, composing $\Psi$ with the symbol map gives a Lie superalgebra homomorphism
         $$\sigma\circ\Psi:\Gamma(M, E^{\mathfrak{su}})\rightarrow
          \mathrm{Symb}({\textstyle\bigwedge^\bullet T^*, \bigwedge^\bullet T^*}).$$

    \end{thm}

\bigskip

       We call this the \textit{super hard Lefschetz action} for
      $\mathbb{K}$-manifolds.

     When   $E^{\mathfrak{su}}$ is trivial, we can take constant
     sections of $E^{\mathfrak{su}}$ and obtain a Lie superalgebra
     action of $\mathfrak{su}_\mathbb{K}(1,1)_{sup}$ on
     $\Omega^\bullet(M)$.
  This happens when the holonomy group of $M$ is inside $SO(n), U(n)$
  or  $Sp(n)$. When $M$ is compact,  the $\mathfrak{su}_\mathbb{K}(1,1)_{sup}$
    action on $\Omega^\bullet(M)$ descends to the cohomology
    $H^*(M)$ by Hodge theory, for which only $\mathfrak{su}_\mathbb{K}(1,1)$ acts non-trivially on $H^*(M)$.
      Our results apply equally well for every normed algebra.
      However, it is more involved to describe precisely the
      algebraic relations for the super hard Lefschetz action for
      $\mathbb{O}$-manifolds due to the non-associative nature of
      $\mathbb{O}$ (see Theorem \ref{tb} for details).

      For Calabi-Yau manifolds $M$, the ``mirror" of the hard
      Lefschetz action should give us another
      $\mathfrak{sl}(2,\mathbb{R})$-action,  at least in the
      semi-flat limit.  This means that the holonomy group of the Calabi-Yau manifolds can be reduced from $SU(n)\subset  GL(n, \mathbb{C})$ to
      $SU(n)\cap GL(n,\mathbb{R})= SO(n)$ (see Definition \ref{semi}). For instance, $T^n$-invariant Calabi-Yau manifolds \cite{leung2} are examples of such.
       In this circumstance, the hard Lefschetz action and its mirror action  combine together to form a
      $\mathfrak{so}(2,2)$-action on differential forms on $M$
      \cite{leung2}. We can  adapt our method easily to this case
      and obtain an enlarged super hard Lefschetz action for
      semi-flat Calabi-Yau and hyperk\"ahler manifolds.  For hyperk\"ahler manifolds, semi-flatness means that the holonomy group   can be reduced from $Sp(n)$ to
      $Sp(n)\cap GL(n,\mathbb{C})= SU(n)$ (see Definition \ref{semi}). Examples of such include  $T^n$-invariant hyperk\"ahler manifolds \cite{leung2}.
      For   $\mathbb{K}=\mathbb{C}$ or
      $\mathbb{H}$, we write $\mathbb{K}'=\mathbb{R}$ or $\mathbb{C}$  respectively, and we have
  \begin{thm}
       Suppose that $M$ is a semi-flat $\mathbb{K}$-manifold with
       $\mathbb{K}$ being $\mathbb{C}$ or $\mathbb{H}$. Then there
       is a natural $\mathfrak{su}_{\mathbb{K}'}(2,2)_{sup}$ action,
       extending the super hard Lefschetz $\mathfrak{su}_{\mathbb{K}}(1,1)_{sup}$ action,
       on the space of differential forms on $M$ via differential
       operators of order at most two.
     \end{thm}

 This paper is organized as follows. In section 2, we  construct the Lie superalgebra
$\mathfrak{su}_\mathbb{K}(1,1)_{sup}$-bundle $E^{\mathfrak{su}}$
over a $\mathbb{K}$-manifold and introduce a (Lie superalgebra)
bundle morphism $\iota$. In section 3, we construct differential
operators via spin actions, apply them to $\mathbb{K}$-manifolds and
prove our main theorems. In section 4, we obtain
$\mathfrak{so}(2,2)_{sup}$ (resp. $\mathfrak{su}(2,2)_{sup}$) action
on  differential forms on semi-flat Calabi-Yau (resp. hyperk\"ahler)
manifolds. Finally in the appendix, we interpret the differential
operators we constructed in terms of the usual ones.

\section{Lie superalgebra bundles over $\mathbb{K}$-manifolds}

  In this section, we first introduce   the notion of
  a $\mathbb{K}$-manifold in terms of its holonomy group  $G_\mathbb{K}$.   We
  then introduce a Lie superalgebra $\mathfrak{su}_\mathbb{K}(1,1)_{sup}$ and
  construct  a  $\mathfrak{su}_\mathbb{K}(1,1)_{sup}$-bundle $E^{\mathfrak{su}}$ over  any $\mathbb{K}$-manifold.
  Finally, we  show that there
  exists another Lie superalgebra  bundle $E$ over any $\mathbb{K}$-manifold, and we
  introduce a  natural bundle morphism  $\iota:
  E^{\mathfrak{su}}\rightarrow E$.

\subsection{$G_\mathbb{K}$ and $\mathbb{K}$-manifolds}

  A normed algebra $\mathbb{K}$  is  a finite dimensional real algebra
 with unit 1 and a norm $\|\cdot\|$ satisfying $\|a\cdot
 b\|=\|a\|\cdot\|b\|$ for any $a, b\in\mathbb{K}$. It is a classical fact
 that  $\mathbb{K}$ is exactly (isomorphic to) one of the following four algebras: the real $\mathbb{R}$, the
complex $\mathbb{C}$, the quaternion $\mathbb{H}$ and the octonion
$\mathbb{O}$.

   For  $m=n\cdot \dim_\mathbb{R}\mathbb{K}$, where $n=1$ if $\mathbb{K}=\mathbb{O}$, we can identify  $V=\mathbb{R}^m$
   with $\mathbb{K}^n$. The standard metric on $V$
   gives an inner product on $\mathbb{K}^n$ satisfying $g(x\cdot
   \alpha, y\cdot \alpha)=g(x, y)\|\alpha\|^2$ for any $x,y\in V$ and
   $\alpha\in\mathbb{K}$.

 \begin{defn} \label{twist}
  A twisted isomorphism $\phi$ of $V$ is a $\mathbb{R}$-isometry  $\phi$ of $V$
  such that there exists $\theta\in SO(\mathbb{K})$ with the
  property
   $ \phi(x\alpha)=\phi(x)\theta(\alpha)$ for any $x\in
              V$ and any $\alpha\in\mathbb{K}$.
   $\phi$ is called special if it preserves the $``\mathbb{K}$-orientation" in
 terms of $``\lambda_{\mathbb{K}}(\phi)$" as
   defined in \cite{leung}.

   We denote by  $G_{\mathbb{K}}(n)$ (resp. $H_{\mathbb{K}}(n)$) the group of (resp. special) twisted isomorphisms of
   $V$.
 \end{defn}

 \begin{defn}  A  Riemannian
   manifold $(M, g)$ is called a (resp. special) $\mathbb{K}$-manifold,  if the holonomy group
   of its Levi-Civita connection  is a subgroup of $G_\mathbb{K}(n)$ (resp. $H_\mathbb{K}(n)$) with $m=\dim M=n\cdot \dim\mathbb{K}$.
  \end{defn}

  From the viewpoint of normed algebras, (non-locally symmetric) Riemannian manifolds with various
  holonomy groups are  classified as follows \cite{leung}.

 \begin{center}
   \begin{tabular}{|c|l|l|}
     \hline

                                                      & $G_{\mathbb{K}}(n)$  &  $H_{\mathbb{K}}(n)$\\
            \raisebox{1.5ex}[0pt]{$\mathbb{K}$}       & ($\mathbb{K}$-manifolds) & (Special  $\mathbb{K}$-manifolds)
            \\\hline\hline
                                                         & $O(n)$  &  $SO(n)$\\
        \raisebox{1.5ex}[0pt]{$\mathbb{R}$}           & (Riemannian manifolds) & (Oriented Riemannian manifolds)  \\
     \hline
                                                      & $U(n)$  &  $SU(n)$\\
         \raisebox{1.5ex}[0pt]{$\mathbb{C}$}          & (K\"ahler manifolds) & (Calabi-Yau manifolds)  \\
     \hline
                                                      & $Sp(n)Sp(1)$  &  $Sp(n)$\\
          \raisebox{1.5ex}[0pt]{$\mathbb{H}$}         & (Quaternionic-K\"ahler manifolds) & (Hyperk\"ahler manifolds) \\
     \hline
                                                      & $Spin(7)$  &  $G_2$\\
          \raisebox{1.5ex}[0pt]{$\mathbb{O}$}       & (Spin(7)-manifolds) & ($G_2$-manifolds) \\
     \hline

   \end{tabular}
  \end{center}

 \bigskip

  In this paper,  we denote  $G_\mathbb{K}(n)$ (resp. $H_\mathbb{K}(n)$) by
           $G_\mathbb{K}$ (resp. $H_\mathbb{K}$) whenever the dimension is well understood.

  \subsection{$\mathfrak{su}_\mathbb{K}(1,1)_{sup}$-bundles over $\mathbb{K}$-manifolds}
   Let $\mathbb{K}$ be a normed algebra, and $\mathrm{Mat}(2, \mathbb{K})$ be $2\times 2$ matrices with
   entries  in $\mathbb{K}$.

  \subsubsection{\label{sec su11}$ \mathfrak{su}_\mathbb{K}(1,1)_{sup}$}
  Each  matrix $A\in\mathrm{Mat}(2, \mathbb{K})$
   induces a real endomorphism $\phi_A:
  \mathbb{K}^2\rightarrow \mathbb{K}^2; u=(u_1,u_2)\mapsto
  \phi_A(u)=uA^\star$, where  $A^\star\triangleq (\overline{A_{ij}} )^T$.
    Denote $(\mathbb{K}^2, \check q)$ by $\mathbb{K}^{1,1}$,
   where $\check q$ is the quadratic form  of type
 $(\dim_\mathbb{R}\mathbb{K}, \dim_\mathbb{R}\mathbb{K})$ defined by $\check q(u, v)\triangleq\mathrm{Re}\big({1\over 2}(u_1\bar v_2+ u_2
 \bar  v_1)\big)$ for any $u, v\in\mathbb{K}^2$.

  Since $\mathbb{H}$ is non-commutative and $\mathbb{O}$ is the
  worst   for its non-associativity, it
  is a little tricky to  define $\mathfrak{sl}(2, \mathbb{K})$ uniformly.
  Following \cite{baez}, we define $\mathfrak{sl}(2, \mathbb{K})$ to
  be
  the real Lie algebra of operators on $\mathbb{K}^2$ generated by
  $\{\phi_A~|~A_{11}+A_{22}=0,\,\, A=\big(A_{ij}\big)\in\mathrm{Mat}(2,\mathbb{K})\}$.
 And we use the
 following notations:
       \begin{align*}
&\mathfrak{su}_\mathbb{K}(1,1) \triangleq\{\phi\in\mathfrak{sl}(2,
\mathbb{K})~|~
   \check q(\phi(u), v)+\check q(u, \phi(v))=0, \forall u,v\in
   \mathbb{K}^2\};\\
    &\mathfrak{su}_\mathbb{K}(1,1)_{sup} \triangleq
 \mathfrak{su}_\mathbb{K}(1,1)\oplus  \mathbb{K}^{1,1} \oplus
 \mathbb{R}.
 \end{align*}

 In fact, $\mathfrak{sl}(2, \mathbb{K})$ and $\mathfrak{su}_\mathbb{K}(1, 1)$ are isomorphic to classical Lie algebras
 below (see the appendix for more details).

 \begin{center}
   \begin{tabular}{|c||c|c|c|c|}
     \hline
     $\mathbb{K}$ & $\mathbb{R}$ & $\mathbb{C}$ & $\mathbb{H}$ & $\mathbb{O}$
     \\\hline
      $\mathfrak{sl}(2, \mathbb{K})$ & $\mathfrak{so}(2,1)$ & $\mathfrak{so}(3,1)$ & $\mathfrak{so}(5,1)$ & $\mathfrak{so}(9,1)$
      \\\hline
     $\mathfrak{su}_\mathbb{K}(1,1)$ & $\mathfrak{so}(1,1)$ & $\mathfrak{so}(2,1)$ & $\mathfrak{so}(4,1)$ & $\mathfrak{so}(8,1)$ \\
     \hline
   \end{tabular}
  \end{center}
  \bigskip
  Furthermore,  $\mathfrak{su}_\mathbb{K}(1,1)_{sup}$ is naturally a Lie superalgebra  because of the following
   remark.
 \begin{remark}\label{ra}  Let $Q$ be a quadratic form on a real vector space $W$, and let
    $\mathfrak{a}$ be a Lie subalgebra of $\mathfrak{so}(W, Q)$.
   Then  $\mathfrak{a}\oplus W\oplus
    \mathbb{R}$ is naturally a Lie superalgebra  with  the following super Lie bracket:
     $\forall \phi, \psi\in \mathfrak{a}, \forall u, v\in W, \forall a, b\in\mathbb{R}$,
   $$[\phi+a, \psi+b]=\phi\psi-\psi\phi,\quad
      [u, v]= -2 Q(u, v),\quad  [\phi+a, u]=\phi(u).$$
 \end{remark}

 \subsubsection{\label{sec gk}$ \mathfrak{su}_\mathbb{K}(1,1)_{sup}$-bundles}
   Let $(M, g)$ be a $\mathbb{K}$-manifold.
   Since Hol$(g)\subset G_\mathbb{K}$,
      its frame bundle can be reduced to a principal $G_\mathbb{K}$-bundle
      $P_{G_\mathbb{K}}$

  By Definition \ref{twist},  there exists a unique $\theta\in SO(\mathbb{K})$ associated
        to $\phi \in  G_\mathbb{K}$. In fact, it induces an action $\Phi$ of $G_\mathbb{K}$ on $\mathbb{K}^{1,1}$
         by
        $\phi\cdot u\triangleq(\theta(u_1), \theta(u_2))$ for any  $u\in
        \mathbb{K}^{1,1}$. It is easy to show that
           $\Phi(G_\mathbb{K})\subset SO(\mathbb{K}^2, \check q)$
           and that $Ad\circ  \Phi$  preserves the Lie subalgebra $\mathfrak{su}_\mathbb{K}(1,1)\subset  \mathfrak{so}(\mathbb{K}^2, \check q)$.
         Therefore, $\Phi$ induces an action $Ad\circ  \Phi$ of
         $G_\mathbb{K}$ on $\mathfrak{su}_\mathbb{K}(1,1)$.
        We take the trivial action of $G_\mathbb{K}$ on
        $\mathbb{R}$, and simply denote by $\Phi$ all these actions. Hence, there exist  the following associated bundles over the  $\mathbb{K}$-manifold $M$:
     $$E^{\mathfrak{su}}_0 \triangleq   P_{G_\mathbb{K}}\times{}_{\Phi}
       \mathfrak{su}_\mathbb{K}(1,1),\quad     E^{\mathfrak{su}}_1 \triangleq
      P_{G_\mathbb{K}}\times{}_{\Phi} \mathbb{K}^{1,1},\quad
       E^{\mathfrak{su}}_2\triangleq     P_{G_\mathbb{K}}\times{}_{\Phi} \mathbb{R}.
             $$
    Note  that $\Phi$ preserves the super Lie bracket of $\mathfrak{su}_\mathbb{K}(1,1)_{sup}$, we have
  \begin{prop}
    There   exists a Lie superalgebra bundle
        $E^{\mathfrak{su}}= E^{\mathfrak{su}}_0\oplus  E^{\mathfrak{su}}_1 \oplus E^{\mathfrak{su}}_2$ over any  $\mathbb{K}$-manifold
        $M$ with fiber $\mathfrak{su}_\mathbb{K}(1,1)_{sup}$.
  \end{prop}
   \begin{example}\label{etrivial}
         The action of  $G_\mathbb{K}$ (resp. $H_{\mathbb{K}}$)  on
         $\mathfrak{su}_\mathbb{K}(1,1)_{sup}$ is trivial, if and
         only if $\mathbb{K}=\mathbb{R}$ or $\mathbb{C}$ (resp. $\mathbb{K}=\mathbb{R}, \mathbb{C}$ or
         $\mathbb{H}$).   Therefore,
         $E^{\mathfrak{su}}$ is  trivial, if and only if  Hol$(g)\subset O(n), U(n)$ or
         $Sp(n)$.
       \end{example}

\subsection{Lie superalgebra bundle morphisms over $\mathbb{K}$-manifolds}

\subsubsection{\label{sec q}$\mathcal{L}$-bundles over $\mathbb{K}$-manifolds}

  Let $g$ be an inner product on a real vector space $V$,
       and let $Q$ be the natural quadratic form on $W=V\oplus V^*$ given by
  $$Q(X+\xi, Y+\eta)\triangleq {\eta(X)+\xi(Y)\over 2}$$
  for any $X, Y\in V $ and any $\xi,\eta\in V^*$.
   It induces a quadratic form $\hat Q$ 
   on Hom$(V^*, W)\cong V\otimes   W$ given by $\hat Q(v_1\otimes w_1, v_2\otimes w_2)\triangleq
   g(v_1, v_2)Q(w_1,w_2)$. Note that the induced action of
   $\mathfrak{so}(W, Q)$ on $V\otimes W$ preserves $\hat Q$.
   Hence, it follows from Remark \ref{ra} that  $$\mathcal{L}\triangleq\mathfrak{so}(W, Q)\oplus
   \mathrm{Hom}(V^*, W)\oplus \mathbb{R}$$ is naturally a Lie superalgebra.

     Let $(M, g)$ be a $\mathbb{K}$-manifold of real dimension $m$.
     The natural action of $G_\mathbb{K}\subset O(m)$ on $V=\mathbb{R}^m$
     induces actions on
       $\mathfrak{so}(W, Q)$, $\mathrm{Hom}(V^*, W)$ and $\mathbb{R}$ respectively in the standard way, which we
       also denote by $\Phi$. Hence, there exist the following associated vector bundles over $M$:
          $$E_0 \triangleq    P_{G_\mathbb{K}}\times{}_{\Phi}\mathfrak{so}(W, Q),
                                       \quad E_1\triangleq P_{G_\mathbb{K}}\times{}_{\Phi}\mathrm{Hom}(V^*, W),
                                       \quad E_2\triangleq
                                       P_{G_\mathbb{K}}\times{}_{\Phi}\mathbb{R}.\,\,$$
    In fact,  $E_0=\bigwedge^2( T\oplus T^*), E_1=\mathrm{Hom}(T^*,
T\oplus T^*)$ and $E_2$ is a trivial line bundle. From the above
discussion, we have the following proposition.

    \begin{prop}
    There   exists a natural Lie superalgebra bundle
        $E=E_0\oplus E_1\oplus E_2$  over any  $\mathbb{K}$-manifold
        $M$ with fiber $\mathcal{L}$.
  \end{prop}
   \subsubsection{\label{sec q2}Lie superalgebra bundle morphisms}

      Let $(M, g)$ be a $\mathbb{K}$-manifold of real dimension $m$.
      Note that $TM=P_{G_\mathbb{K}}\times_\Phi
      V$, where   $V=\mathbb{R}^m$ is identified with
      $\mathbb{K}^n$.

      There
      is a natural monomorphism of Lie algebras
              $$\iota:  \mathfrak{su}_\mathbb{K}(1,1)\hookrightarrow
              \mathfrak{so}(W, Q)$$
       defined
            as  follows. If $\mathbb{K}$ is associative,
     $\iota(L)$ is
      given by the following procedure
      $$\iota(L):  V\oplus   V^*\overset{\psi_1}{\longrightarrow}
                \mathbb{K}^n\oplus\mathbb{K}^n
                 \overset{\psi_2}{\longrightarrow} \mathbb{K}^n   \otimes_\mathbb{K}  \mathbb{K}^{1, 1}
                  \overset{\mathrm{Id}\otimes L}{-\!\!\!-\!\!\!-\!\!\!\longrightarrow} \mathbb{K}^n   \otimes_\mathbb{K}  \mathbb{K}^{1, 1}
               \overset{(\psi_2\circ\psi_1)^{-1}}{-\!\!\!-\!\!\!-\!\!\!-\!\!\!-\!\!\!\longrightarrow}V\oplus        V^*,$$
    where $\psi_1(\mathbf{x}, \mathbf{\xi})=(\mathbf{x}_\mathbb{K}, \mathbf{\xi}_\mathbb{K})$ is
    the natural identification of  $V\oplus V^*$ with $\mathbb{K}^n\oplus \mathbb{K}^n$
    and  $\psi_2(\mathbf{x}_\mathbb{K}, \mathbf{\xi}_\mathbb{K})=
                 \mathbf{x}_\mathbb{K}\otimes (1,0)+\mathbf{\xi}_\mathbb{K}\otimes(0, 1) $ is the natural isomorphism.
            If $\mathbb{K}$ is not associative, in which case we note that  $\mathbb{K}=\mathbb{O}$ and $V\cong \mathbb{O}$,
            then  $\iota(L)$ is given by the following procedure
               $\iota(L):  V\oplus   V^*\overset{\psi_1}{\longrightarrow}
                \mathbb{O}\oplus\mathbb{O}
                     \overset{  L}{\longrightarrow} \mathbb{O}\oplus\mathbb{O}
               \overset{\psi_1 ^{-1}}{-\!\!\!-\!\!\!\longrightarrow}V\oplus        V^*.$

   There is also a natural inclusion
           $\iota: \mathbb{K}^{1, 1}\rightarrow \mathrm{Hom}(V^*,
           W); u=(u_1, u_2)\mapsto \iota(u)$ as defined by the
           following procedure
              $\iota(u): V^*\overset{\psi_1}{\longrightarrow}\mathbb{K}^n \overset{\psi_u}{\longrightarrow}
             \mathbb{K}^n\oplus\mathbb{K}^n
             \overset{\psi_1 ^{-1}}{-\!\!\!-\!\!\!\longrightarrow}V\oplus
             V^*, $
     where $\psi_u(\xi_\mathbb{K})=(\xi_\mathbb{K}u_1,
      \xi_\mathbb{K}u_2)$.  Together with the map $\iota: \mathbb{R}\rightarrow \mathbb{R}$
      given by  $\iota(a)\triangleq ma$, we obtain a map
       $$\iota:\mathfrak{su}_\mathbb{K}(1,1)_{sup}\rightarrow
       \mathcal{L}.$$
       Note that
      the action of $G_\mathbb{K}$   on $\mathcal{L}$
      via the inclusion into $O(m)$ is standard. Then it is straightforward to
      get the following lemma, the proof of which we omit.

 \begin{lemma}\label{lemaa}
   \begin{enumerate}
      \item $G_\mathbb{K}$ preserves the subspace  $\iota(\mathfrak{su}_\mathbb{K}(1,1)_{sup})$ of $\mathcal{L}$.
      \item If $\mathbb{K}$ is associative, $\iota:\mathfrak{su}_\mathbb{K}(1,1)_{sup}\hookrightarrow \mathcal{L}$
                is an injective morphism of  Lie superalgebras.
      If $\mathbb{K}=\mathbb{O}$, $\iota(\mathfrak{su}_\mathbb{K}(1,1))\cdot \iota(\mathbb{O}^{1,1})
                           =\mathrm{Hom}(V^*, V\oplus V^*)$.

    \end{enumerate}
 \end{lemma}

 Thus, there is an
  action  of  $G_\mathbb{K}$ on
  $\mathfrak{su}_\mathbb{K}(1,1)_{sup}$ by viewing it as the subspace
  $\iota(\mathfrak{su}_\mathbb{K}(1,1)_{sup})$ of $\mathcal{L}$. In fact, this action is exactly the same
  as the $G_\mathbb{K}$ action as introduced in section \ref{sec gk}.
   Since all the actions come  out in the standard way, we denote
   all of them by the same notation $\Phi.$
    Consequently, we have an induced vector bundle embedding $$\iota: E^{\mathfrak{su}}\hookrightarrow   E.$$
   Following from Lemma
    \ref{lemaa}, we have
    \begin{prop}
      Let $M$ be a $\mathbb{K}$-manifold.  If $\mathbb{K}$ is associative, then $\iota: E^{\mathfrak{su}}\hookrightarrow
      E$ is an injective Lie superalgebra bundle morphism.
    \end{prop}

    For a bundle   $B$ over $M$, we denote the space of sections as $\Gamma(M, B)$, or simply $\Gamma(B)$.
 We denote   by
   $\check q$ the bilinear form on $\Gamma(E^{\mathfrak{su}}_1)$
   induced from the quadratic form $\check q$ on $\mathbb{K}^{1,1}$. And we  denote  by $\iota$ the
   induced inclusion   $\Gamma(E^{\mathfrak{su}})\rightarrow\Gamma(E)$ from $\iota: \mathfrak{su}_\mathbb{K}(1,1)_{sup}\rightarrow \mathcal{L}$.

 \section{Lie superalgebra bundle action on forms}
   In this section, we construct differential operators of
    order zero, one and two on
 differential forms on a $\mathbb{K}$-manifold $M$, and compute (some of) their supercommutators. Using these, we proceed
 to obtain
the main result  of this paper, namely  there is a natural Lie
superalgebra homomorphism $\sigma\circ\Psi:
\Gamma(E^{\mathfrak{su}})\longrightarrow
\mathrm{Symb}(\bigwedge^\bullet V^*, \bigwedge^\bullet V^*)$, when
$\mathbb{K}$ is associative (i.e. $\mathbb{K}=\mathbb{R},
\mathbb{C}$ or $\mathbb{H}$).

\subsection{\label{sec spin}Spin action on $\bigwedge^\bullet V^*$}
 Let $V$ be a real vector space. The vector space $W=V\oplus V^{\ast }$ has a natural quadratic form $Q$ and a natural spin structure \cite{mar}.
The spinor
representation $S$ of $Spin\left( W,Q\right) $ can be naturally identified with $%
\bigwedge^\bullet V^*$ using the following linear action of $W$ on
$\bigwedge^\bullet V^*$:
  $$(X+\xi)\cdot \varphi \triangleq \xi\wedge \varphi - i_X(\varphi),\;
     \mbox{where}\,\, X\in V, \xi\in V^*\;\mbox{and}\; \varphi\in  {\textstyle\bigwedge}^\bullet V^*.$$
 Recall that $Spin\left( W,Q\right) $ is a double cover of $SO\left(
W,Q\right) $ and the induced isomorphism on the Lie algebra level is
given by (cf. \cite{LM}):%
\begin{eqnarray*}
\mathrm{ad} &:&\mathfrak{spin}\left( W,Q\right) \overset{\cong
}{\longrightarrow }\mathfrak{so}\left(
W,Q\right);  \\
&&x \mapsto
     \mathrm{ad}(x),\mbox{ where }\, \mathrm{ad}(x):  W\rightarrow W\,;
 \mathrm{ad}(x)(w)= x w-w x \text{.}
\end{eqnarray*}%
 Thus given a metric on $V$, we can identify  $\mathfrak{so}\left( V\right)$ with a Lie subalgebra of $\mathfrak{spin}(W, Q)$ via
  $\mathrm{ad}^{-1}\circ \psi_4$, where $\psi_4$ is the diagonal embedding of
$\mathfrak{so}(V)$ into $\mathfrak{so}(W, Q)$. Using this
identification, one can show that this spin action of
$\mathfrak{spin}\left(
W,Q\right) $ on $S=\bigwedge^\bullet V^*$ restricts to the usual action of $%
\mathfrak{so}\left( V\right) $ on $\bigwedge^\bullet V^*$. Globally
over a manifold,  $\bigwedge^\bullet T^*$ can be identified as a
             spinor bundle  of $T\oplus T^*$ \cite{mar}.

\subsection{\label{op0 sec}Zeroth order operators}

    Let $(M, g)$ be a  $\mathbb{K}$-manifold. From now on, we always assume that $M$ is
    orientable (in the usual sense).   Then Hol$(g)\subset
    G_\mathbb{K}^\circ$, where $G_\mathbb{K}^\circ$ is the connected component of $G_\mathbb{K}$.
    We note that $G_\mathbb{K}^\circ=G_\mathbb{K}$ if $\mathbb{K}\neq \mathbb{R}$.

     Let $\mathcal{S}=\bigwedge^\bullet T^*$. Denote by
    $\mathrm{Diff}_k(\mathcal{S} , \mathcal{S})$ the space of differential operators of order $k$
      on $\Gamma(\mathcal{S})=\Omega^\bullet(M)$, and put $\mathrm{Diff}(\mathcal{S} , \mathcal{S})=\bigoplus_{k=0}^\infty
      \mathrm{Diff}_k(\mathcal{S} , \mathcal{S})$. In particular,
        $\mathrm{Diff}_0(\mathcal{S} , \mathcal{S})=\Gamma(\mathrm{End}(\mathcal{S}))$.
        With the natural isomorphism $\mathrm{ad}$ and the spinor representation  as mentioned in section \ref{sec spin},
        together with the natural inclusion $\iota:\Gamma(E^{\mathfrak{su}}_0)\hookrightarrow
   \Gamma(E_0)$, we obtain the following natural maps.
    \begin{defn}
       Define  $\Psi: \Gamma(E_0)\rightarrow \mathrm{Diff}_0(\mathcal{S},
         \mathcal{S})$ by $\Psi(x)=\rho_x$, where $\rho_x$ is
         defined as follows:  $(\rho_x\varphi)(p)=\mathrm{ad}^{-1}(x(p))\cdot
         \varphi(p)$
         for any $\varphi\in \Gamma(\mathcal{S})$ and any
         $p\in M$.
    \end{defn}

     \begin{defn}
         Define $\Psi_\iota: \Gamma(E^{\mathfrak{su}}_0)\rightarrow \mathrm{Diff}_0(\mathcal{S},
         \mathcal{S})$ by $
         \Psi_\iota(x)=\rho_{\iota(x)}$. We simply denote $\Psi_\iota$ (resp. $\rho_{\iota(x)}$) by
        $\Psi$ (resp. $\rho_x$).
    \end{defn}
     In \cite{li},
    the second author has studied the cases Hol$(g)\subset SO(n), U(n)$ and $Sp(n)$.
     We will restate the results in the appendix.

   Let $V=\mathbb{R}^m$ and $S=\bigwedge^\bullet
    V^*$. Since Hol$(g)\subset  G_\mathbb{K}^\circ$, the frame bundle of $M$ can be reduced to a
     $ G_\mathbb{K}^\circ$-bundle $P$ such that
    $\mathcal{S}=P\times_\Phi   S$.
    Note that there is a canonical bijection between
    $\Gamma(\mathcal{S})$ and the $ G_\mathbb{K}^\circ$-invariant sections $\Gamma(P,
    S)^{ G_\mathbb{K}^\circ}$ \cite{joy}. In order to obtain an operator on $\Gamma(\mathcal{S})$,
    it is enough to construct an operator on $\Gamma(P, S)^{ G_\mathbb{K}^\circ}$.

\begin{example}\label{erho}
   Let $\{f_j\}_{j=1}^m$ be the standard basis of $V$, and $\{f^j\}$ be the dual basis.
   Then  $\mathrm{ad}^{-1}(\psi_4(\mathfrak{so}(m)))=\mbox{Span}\{e_{i+m}e_{j+m}-e_ie_j~|~ 1\leq i< j\leq m\}\subset \mathfrak{spin}(W, Q)$,
  where $e_j=f^j+f_j, e_{j+m}=f^j-f_j, j=1, \cdots, m$.

   Note that    $\nu=e_1\cdots e_m\in Cl(W,
  Q)$ and that
$(e_{i+m}e_{j+m}-e_ie_j)\nu=\nu(e_{i+m}e_{j+m}-e_ie_j)$ for any
 $1\leq i< j\leq m$. As mentioned in section \ref{sec spin}, the spin action of  $\mathrm{ad}^{-1}(\psi_4(\mathfrak{so}(m)))$
 equals the usual action of $\mathfrak{so}(m)$ on $S$. Hence,
  the natural action  $\mathbb{R}\nu\rightarrow \mathrm{End}(S)$ commutes
  with the standard action  of the connected compact group $ G_\mathbb{K}^\circ$ on $S$. Hence,   $\nu$ provides an operator
   on $\Gamma(P, S)^{ G_\mathbb{K}^\circ}$, and therefore it induces a global    operator
      $\rho_\nu$ of order zero. In fact, $\rho_\nu|_{\Omega^r(M)}=(-1)^{mr+{r(r-1)\over
          2}}\star|_{\Omega^r(M)}.$
  \end{example}

\subsection{\label{op1 sec}First order operators}
  Recall that
$\Gamma(E_1)=\Gamma\big(\mathrm{Hom}(T^*, T\oplus
 T^*)\big)$, and  that the
Levi-Civita connection
 $\nabla$ of $M$ is a $
 G_\mathbb{K}^\circ$-connection.
 With the help of $\nabla$, we obtain the following natural map.
  \begin{defn}
       Define  $\Psi: \Gamma(E_1)\rightarrow \mathrm{Diff}_1(\mathcal{S},
         \mathcal{S})$ by $\Psi(u)=D_u$, where $D_u$ is the first order
         operator given by  composition of the following maps
 $$D_u:\,\,\,\Gamma(\mathcal{S})\overset{\nabla}{\longrightarrow}\Gamma(T^*\otimes
\mathcal{S})\overset{ u}{\longrightarrow}
                 \Gamma((T\oplus T^*)\otimes \mathcal{S})
      \overset{\mbox{ Clifford
     product}}{-\!\!\!-\!\!\!-\!\!\!-\!\!\!-\!\!\!-\!\!\!-\!\!\!-\!\!\!-\!\!\!-\!\!\!-\!\!\!-\!\!\!\longrightarrow}\Gamma(\mathcal{S}).$$
   \end{defn}

   By the natural identification of $Cl=Cl(T\oplus T^*, Q)$ with $\bigwedge^\bullet (T\oplus
    T^*)$, $D_u$ can also act on  $\Gamma(Cl)$
   through a similar procedure:
  $$D_u:\,\,\,\Gamma(Cl)\overset{\nabla}{\longrightarrow}\Gamma(T^*\otimes
Cl)\overset{ u}{\longrightarrow}
                 \Gamma((T\oplus T^*)\otimes Cl)
                  \overset{\mbox{ Clifford
     product}}{-\!\!\!-\!\!\!-\!\!\!-\!\!\!-\!\!\!-\!\!\!-\!\!\!-\!\!\!-\!\!\!-\!\!\!-\!\!\!-\!\!\!\longrightarrow}\Gamma(Cl).$$
 In particular, for any $x\in \Gamma(E_0)=\Gamma(\bigwedge^2 (T\oplus
 T^*))$, $D_u x$ is meaningful, where we regard $x$ as a section in
 $\Gamma(Cl)$ via ad${}^{-1}$. Note that
 $\nabla_X(s\cdot\varphi)=(\nabla_Xs)\cdot
 \varphi+s\cdot\nabla_X\varphi$, for any $X\in\Gamma(TM)$, any
 $s\in\Gamma(Cl)$ and any $\varphi\in\Gamma(\mathcal{S})$. We have

\begin{prop}\label{psu}
     For any  $x\in\Gamma(E_0)$ and any $u\in
\Gamma(E_1)$,
 $$ \rho_x\circ D_u-D_u\circ\rho_x=D_{x\cdot u} -D_ux.$$
  \end{prop}

\begin{proof}
  It is sufficient to prove it locally. Let $U$ be a coordinate
  chart with local coordinate $(y_1, \cdots, y_m)$. Denote $\nabla_{\!\!{\partial \over\partial y_j}}$ by $\nabla_{\!\!j}$.
   For any
  $\varphi\in\Gamma(U, \mathcal{S})$,   we have
\\$(\rho_x\circ D_u)\varphi =\mathrm{ad}^{-1}(x)\cdot\sum_{j=1}^m
u(dy^j)\cdot\nabla_{\!\!j}\varphi$, and we have
\\$(D_u\circ\rho_x)\varphi
=\sum\limits_{j=1}^m
u(dy^j)\cdot\nabla_{\!\!j}(\mathrm{ad}^{-1}(x)\cdot
\varphi)=\sum\limits_{j=1}^m u(dy^j)\cdot\big((\nabla_{\!\!j}
\mathrm{ad}^{-1}(x))\cdot \varphi
         +\mathrm{ad}^{-1}(x)\cdot\nabla_{\!\!j}\varphi\big).$

     \begin{align*}{}\mathrm{Hence},&\,\,\,\rho_x\circ D_u\varphi-D_u\circ\rho_x\varphi\\
              =&\big(\sum_{j=1}^m (\mathrm{ad}^{-1}(x)\cdot u(dy^j)-u(dy^j)\cdot\mathrm{ad}^{-1}(x) )\cdot\nabla_{\!\!j}\varphi
              \big)    -\sum_{j=1}^m  u(dy^j)\cdot\nabla_{\!\!j} \mathrm{ad}^{-1}(x) \cdot\varphi\\
              =&\big(\sum_{j=1}^m(\mathrm{ad}(\mathrm{ad}^{-1}(x))\cdot u(dy^j) )\cdot\nabla_{\!\!j}\varphi
                  \big)    -(D_ux)\varphi\\
              =&\big(\sum_{j=1}^m(x\cdot u)(dy^j) \nabla_{\!\!j}
              \varphi \big)    -(D_ux)\varphi\\
              =&D_{x\cdot u} \varphi -(D_ux)\varphi.
     \end{align*}
  Hence, $\rho_x\circ D_u-D_u\circ\rho_x=D_{x\cdot u}  -D_ux$.
\end{proof}



                Because of the inclusion $\iota: \Gamma(E^{\mathfrak{su}}_1)\rightarrow
                  \Gamma(E_1)$, we have the following natural map.

 \begin{defn}
        Define  $\Psi_\iota: \Gamma(E^{\mathfrak{su}}_1)\rightarrow \mathrm{Diff}_1(\mathcal{S},
         \mathcal{S})$ by  $
         \Psi_\iota(u)=D_{\iota(u)}$. We simply denote $\Psi_\iota$ (resp. $D_{\iota(u)}$) by
         $\Psi$ (resp. $D_u$).
  \end{defn}

     Note that the action of $G_\mathbb{K}^\circ$ on $\mathbb{R}^{1,1}=\mathbb{R}\epsilon_1\oplus\mathbb{R}\epsilon_2$ is always trivial,
     where  $\epsilon_1=(1, 0), \epsilon_2=(0, 1)\in\mathbb{K}^{1,1}$. Hence,
     $E^{\mathfrak{su}}_1$ has a trivial subbundle $M\times\mathbb{R}^{1,1}$.
     Therefore  the constant section $\epsilon_j$ induces a first order operator $D_{\epsilon_j}, j=1, 2$.
     Moreover, it follows from the observation
     $\nu\cdot(\iota(\epsilon_2)(f^k))\cdot
     \nu^{-1}=(-1)^{m-1}\iota(\epsilon_1)(f^k)$ and the construction of $\rho_\nu$ as in
     Example \ref{erho}   that $  D_{\epsilon_1}=(-1)^{m-1}\rho_{\nu}D_{\epsilon_2}\rho_\nu^{-1}. $

      Because of the use of the Levi-Civita connection, we have
          $D_{\epsilon_2}=\sum_j
                  dy^j\wedge\nabla_{\!\!{\partial\over \partial y_j}}=d$ and  $D_{\epsilon_1}=\sum_j
                 -i_{{\partial\over \partial y_j}}\circ\nabla_{\!\!{\partial\over \partial y_j}}=d^*$  (cf.
                  \cite{LM}). In particular,
                  $D_{\epsilon_1}{}^2=D_{\epsilon_2}{}^2=0$.
                  However, we would rather make the assumption
                  ``$D_{\epsilon_2}{}^2=0$" in Proposition
                  \ref{pde}, for possible application to other
                  cases.

  \subsection{Second order operators}
   For any linear operators $a, b, c$ on $\Gamma(\mathcal{S})$, we   define $\{a,
       b\}\triangleq ab+ba$ and
        $[a, b]\triangleq ab-ba$.
 Clearly,   $ [a, \{b, c\}]=\{[a, b], c\}+\{b, [a, c]\}$.

       Define $\triangle=\{D_{\epsilon_1}, D_{\epsilon_2}\}.$
       Then we have

  \begin{prop}\label{pdd}
       For any $u, v\in \Gamma(E^{\mathfrak{su}}_1)$,  $\{D_u, D_v\}-2\check q(u,v)\triangle$
        is a first order differential operator.
 \end{prop}

   We will give a proof by computing the symbols in the appendix.
   At the moment, we would like to give an extension of
  $\mathfrak{su}_\mathbb{K}(1,1)$. We define   $\mathfrak{u}_\mathbb{K}(1,1)$ to be
    $\mathfrak{su}_\mathbb{K}(1,1)$ itself if $\mathbb{K}\neq
    \mathbb{C}$, and let  $\mathfrak{u}_\mathbb{C}(1,1)\triangleq  \mathfrak{su}_\mathbb{C}(1,1)\oplus \mathbb{R}\phi_A$,
    where $\phi_A\in\mathfrak{so}(\mathbb{K}^2, \check q)$ is as defined in section \ref{sec su11} with
    $A=\sqrt{-1}\cdot \mathrm{I}_2$, the product of $\sqrt{-1}$ and the identity matrix $\mathrm{I}_2\in\mathrm{Mat}(2, \mathbb{C})$.
   Then we have $ \mathfrak{u}_\mathbb{K}(1,1)=\Big\{\Bigg(\begin{array}{cc} \beta_1 & \beta_2\\\beta_3 &-\bar\beta_1
                        \end{array}\Bigg)~\Big|~ \beta_1\in\mathbb{K}, \beta_2,\beta_3\in
                                \mbox{Im}\mathbb{K} \Big\}$, if $\mathbb{K}$ is     associative.
   Furthermore, all the statements  after section \ref{sec su11} that
   involve  $\mathfrak{su}_\mathbb{K}(1,1)$  still hold true if we
   replace $\mathfrak{su}_\mathbb{K}(1,1)$ with
   $\mathfrak{u}_\mathbb{K}(1,1)$.  With this observation, we can  provide another  proof  for the  most relevant case as
  below.

   \begin{prop}\label{pde}
      Suppose $E^{\mathfrak{su}}$ is trivial and $D_{\epsilon_2}{}^2=0$.
      Then for any constant sections $u, v\in \Gamma(E^{\mathfrak{su}}_1)=\Gamma(M\times \mathbb{K}^{1,1})$,
          $$\{D_u, D_v\}=2\check q(u,v) \triangle.$$
 \end{prop}

 \begin{remark}   It follows from Example \ref{etrivial} that  $E^{\mathfrak{su}}$ is trivial
     only if $\mathbb{K}$ is associative.
 \end{remark}

\begin{proof}[Proof of Proposition \ref{pde}]
Because of the decomposition
$\mathbb{K}^{1,1}=\mathbb{R}\epsilon_1\oplus
\mathrm{Im}\mathbb{K}\epsilon_1\oplus
 \mathbb{R}\epsilon_2\oplus\mathrm{Im}\mathbb{K}\epsilon_2$, we can
 write any $u, v\in\mathbb{K}^{1,1}$ as $u=u_{1\mathrm{r}}+u_{{1\mathrm{i}}}+u_{2\mathrm{r}}+u_{2\mathrm{i}}$
 and  $v=v_{1\mathrm{r}}+v_{{1\mathrm{i}}}+v_{2\mathrm{r}}+v_{2\mathrm{i}}$.

\bigskip

 \noindent Case $u, v\in\mathbb{R}\epsilon_1\oplus\mathbb{R}\epsilon_2$:

      As mentioned in section \ref{op1 sec}, $D_{\epsilon_1}=(-1)^{m-1}\rho_{\nu}D_{\epsilon_2}\rho_{\nu}^{-1}$. Hence, it follows
     from $D_{\epsilon_2}{}^2=0$ that   $D_{\epsilon_1}{}^2=0$.
       Note that   $\check q(\epsilon_1, \epsilon_2)={1\over  2}$ and
          $ \check q(\epsilon_1, \epsilon_1)=\check q(\epsilon_2,
          \epsilon_2)=0$.
      For the constant sections     $u, v\in\mathbb{R}\epsilon_1\oplus\mathbb{R}\epsilon_2$, there exist $a_1, a_2, b_1, b_2\in\mathbb{R}$ such that
           \begin{align*}\{ D_{u  },
           D_{v }\}
             &=\{D_{a_1\epsilon_1+a_2\epsilon_2},D_{b_1\epsilon_1+b_2\epsilon_2}\}\\
             &=\{a_1D_{\epsilon_1}+a_2D_{\epsilon_2},b_1D_{\epsilon_1}+b_2D_{\epsilon_2}\}\\
             &=(a_2b_1+a_1b_2)\triangle\\
             &=2\check q(u, v)\triangle.
          \end{align*}

\bigskip

 \noindent Case  $ u \in\mathrm{Im}\mathbb{K}\epsilon_1,  v \in  \mathbb{R}\epsilon_1\oplus\mathrm{Im}\mathbb{K}\epsilon_1\oplus\mathbb{R}\epsilon_2 $:

    Note that for any constant sections $x\in
       \Gamma(E^{\mathfrak{su}}_0)$  and $w\in
       \Gamma(E^{\mathfrak{su}}_1)$, $D_wx=0$ and
       $\iota(x)\cdot\iota(\epsilon_{\ell})=\iota(x\cdot \epsilon_{\ell})$,
       $\ell=1, 2$.  Since $\check q(\epsilon_{\ell},  \epsilon_{\ell})=0$, we have
          \begin{align*} 0=[\rho_x, 0]&=[\rho_x, \{D_{\epsilon_{\ell}}, D_{\epsilon_{\ell}}\}]\\
                                  &=\{[\rho_x, D_{\epsilon_{\ell}}], D_{\epsilon_{\ell}}\}+\{D_{\epsilon_{\ell}},
                                            [\rho_x,  D_{\epsilon_{\ell}}]\}\\
                                  &= 2\{D_{x\cdot {\epsilon_{\ell}}}, D_{\epsilon_{\ell}}\}
                                          \qquad\quad(\mbox{by Proposition
                                          }\ref{psu}).
         \end{align*}
         Again note that $0=\check{q}(x\cdot \epsilon_{\ell}, \epsilon_{\ell})+\check{q}(\epsilon_{\ell}, x\cdot \epsilon_{\ell})
                        =2\check{q}(x\cdot
                       \epsilon_{\ell},\epsilon_{\ell})$. Hence,
                $$\{ D_{x\cdot
                \epsilon_{\ell}}, D_{\epsilon_{\ell}}\}=0=2\check q(x\cdot
                        \epsilon_{\ell}, \epsilon_{\ell})\triangle.$$
          Note that  $u=c\epsilon_1$ for some  $c\in\mathrm{Im}\mathbb{K}$,
           $x\triangleq \left(\begin{array}{cc}0&-c \\0&0\end{array}\right)$
           and $y\triangleq \left(\begin{array}{cc}-c &0\\0&-c \end{array}\right)$
           are  constant sections in $\Gamma(E^{\mathfrak{su}}_0)$ such that $x\cdot
           \epsilon_2=u $ and  $y\cdot
           \epsilon_1=u$.
           Take $b_1, b_2\in\mathbb{R}$ such that $v=b_{1}\epsilon_1+v_{1\mathrm{i}}+b_2\epsilon_2$ where
           $v_{1\mathrm{i}}\in \mathrm{Im}\mathbb{K}\epsilon_1$.
       Then we have

       $\{ D_{u }, D_{b_{1}\epsilon_1}\}=b_1\{ D_{y\cdot \epsilon_1},
       D_{\epsilon_1}\}=0$ and   $\{ D_{u }, D_{b_2 \epsilon_2}\}=b_2\{ D_{x\cdot \epsilon_2},
         D_{\epsilon_2}\}=0.$
  \\Note that $\{D_{\epsilon_2},  D_{v_{1\mathrm{i}}}\}=\{D_{v_{1\mathrm{i}}},
     D_{\epsilon_2}\}=0$  and that   $ \iota(x)\cdot \iota(v_{1\mathrm{i}})=0$, we have
               \begin{align*} \{D_{u }, D_{v_{1\mathrm{i}}}\}=\{D_{x\cdot\epsilon_2}, D_{v_{1\mathrm{i}}}\}&=\{[\rho_x, D_{\epsilon_2}], D_{v_{1\mathrm{i}}}\} \\
                                  &=[\rho_x, \{D_{\epsilon_2},  D_{v_{1\mathrm{i}}}\}]-\{D_{\epsilon_2},   [\rho_x,  D_{v_{1\mathrm{i}}}]\} \\
                                  &= [\rho_x, 0] -\{D_{\epsilon_2},   D_{ \iota(x)\cdot \iota({v_{1\mathrm{i}}})}\} \\
                                  &= 0-\{D_{\epsilon_2},   0\}=0
                \end{align*}
  Since $ u \in\mathrm{Im}\mathbb{K}\epsilon_1$ and $v \in  \mathbb{R}\epsilon_1\oplus\mathrm{Im}\mathbb{K}\epsilon_1\oplus\mathbb{R}\epsilon_2$,
  $\check q(u, v)=0$. Therefore we have,  $$\{D_{u },  D_v\}=\{D_{u },  D_{b_{1}\epsilon_1}\}+\{D_u, D_{v_{1\mathrm{i}}}\}+\{D_u, D_{b_{2}\epsilon_2}\}=0=
          2\check q(u, v)\triangle.$$
\bigskip

 \noindent Case  $ u \in\mathrm{Im}\mathbb{K}\epsilon_2,  v \in  \mathbb{R}\epsilon_1\oplus\mathbb{R}\epsilon_2 \oplus\mathrm{Im}\mathbb{K}\epsilon_2$:

      $u=c\epsilon_2$ for some $c\in\mathrm{Im}\mathbb{K}$. Define  $x\triangleq \left(\begin{array}{cc}0&0\\-c&0\end{array}\right)$
           and $y\triangleq \left(\begin{array}{cc}-c &0\\0&-c
           \end{array}\right)$, and use the same method as above, we
           can show the formula $\{D_u ,  D_v\}=0= 2\check q(u, v)\triangle.$

\bigskip

 \noindent Case  $ u \in\mathrm{Im}\mathbb{K}\epsilon_1,  v \in  \mathrm{Im}\mathbb{K}\epsilon_2 $:

  Take $c_1,
   c_2\in\mathrm{Im}\mathbb{K}$ such that $u=c_1\epsilon_1$ and $v=c_2\epsilon_2$. Then we have the constant sections
       $x\triangleq \left(\begin{array}{cc}0&-c_1\\0&0\end{array}\right)$
           and $y\triangleq \left(\begin{array}{cc}0 &0\\-c_2&0
           \end{array}\right)$  in
           $\Gamma(E^{\mathfrak{su}}_0)$ such that   $x\cdot
           \epsilon_2=u $ and  $y\cdot
           \epsilon_1=v $.
           Note that $[\rho_{x },  \rho_{y}]=\rho_{[x, y]}$ and $x\cdot \epsilon_1=0$, we have
             \begin{align*} \{D_{u }, D_{ v }\}=\{D_{x \cdot\epsilon_2}, D_{
             v}\}
                                  &=\{[\rho_{x }, D_{\epsilon_2}], D_v\} \\
                                  &=[\rho_{x}, \{D_{\epsilon_2},  D_v\}]-\{D_{\epsilon_2},   [\rho_{x },  D_{v}]\} \\
                                  &=-\{D_{\epsilon_2},   [\rho_{x},  [\rho_{y}, D_{\epsilon_1}]]\} \\
                                  &=-\{D_{\epsilon_2},   [[\rho_{x},  \rho_{y}], D_{\epsilon_1}]+[\rho_{y},[\rho_{x}, D_{\epsilon_1}]]\} \\
                                  &=-\{D_{\epsilon_2},   D_{[x, y]\cdot\epsilon_1}+0\}  \\
                                  &=-2\check q(\epsilon_2,   [{x},  y]\cdot\epsilon_1)\triangle \\
                                  &=2\check q(x\cdot\epsilon_2,    y\cdot\epsilon_1)\triangle\\
                                  &=2\check q(u,  v )\triangle
               \end{align*}
         Since the product $\{\cdot, \cdot\}$ is symmetric, the
         formula $\{D_{u }, D_{ v }\}=2\check q(u,  v )\triangle$ also holds true for the remaining cases.
           Hence, we have completed the proof.
\end{proof}

 \subsection{\label{ap sec}Main results}
 Both $\Gamma(E^{\mathfrak{su}})$ and $\Gamma(E)$ have
   induced Lie superalgebra  structures.  From Lemma \ref{lemaa},
   $\iota: \Gamma(E^{\mathfrak{su}})\rightarrow\Gamma(E)$ is no longer a Lie superalgebra morphism when $\mathbb{K}=\mathbb{O}$.
   However,
   $\iota(\Gamma(E^{\mathfrak{su}}_0))\oplus\big(\iota(\Gamma(E^{\mathfrak{su}}_0))\cdot \iota(\Gamma(E^{\mathfrak{su}}_1))\big)\oplus\iota(\Gamma(E^{\mathfrak{su}}_2))$
  is always a super Lie subalgebra of $\Gamma(E)$.

  For any differential operator $D$ of order $k$, its symbol
  $\sigma_k(D)$ is an element in $\mathrm{Symb}_k(\mathcal{S},
  \mathcal{S})=\Gamma\big(M, \mathrm{Sym}^kT^*\otimes \mathrm{Hom}(\mathcal{S},
  \mathcal{S})\big)$ \cite{wells}. This symbol map fits the following
  exact sequence
      $$0\longrightarrow \mathrm{Diff}_{k-1}(\mathcal{S}, \mathcal{S})
      \overset{j}{\longrightarrow}\mathrm{Diff}_{
    k}(\mathcal{S}, \mathcal{S})\overset{\sigma_k}{\longrightarrow}\mathrm{Symb}_{
    k}(\mathcal{S}, \mathcal{S}),$$
 where $j$ is the natural inclusion. Furthermore, $ \mathrm{Symb}(\mathcal{S},\mathcal{S})=\bigoplus_{k=0}^\infty
     \mathrm{Symb}_k(\mathcal{S}, \mathcal{S})$
         has a natural Lie   superalgebra  structure such that
         $$\sigma: \mathrm{Diff}(\mathcal{S}, \mathcal{S})\longrightarrow\mathrm{Symb}(\mathcal{S}, \mathcal{S}) $$
  is a  Lie superalgebra homomorphism.

    Recall that for any section $(x, u)$ in  $\Gamma(E_0)\oplus\Gamma(E_1)$ (resp.
      $\Gamma(E^{\mathfrak{su}}_0)\oplus\Gamma(E^{\mathfrak{su}}_1)$),  we have constructed the associated  differential operator (of order zero
    and one) $(\rho_x, D_u)$ (resp. $(\rho_{\iota(x)}, D_{\iota(u)})$).  Note that both $E^{\mathfrak{su}}_2$ and $E_2$ are trivial line
   bundles, any smooth section $f$ of $E_2$ (resp.
   $E^{\mathfrak{su}}_2$) is a smooth function on $M$.
 Then we
    obtain the  following natural maps.

    \begin{defn}
        Define  $\Psi: \Gamma(E)\rightarrow \bigoplus_{k=0}^2\mathrm{Diff}_k(\mathcal{S},
         \mathcal{S})\subset\mathrm{Diff}(\mathcal{S},
         \mathcal{S})$   by $\Psi(x, u, f)=(\rho_x, D_u, -{1\over\dim M}f\triangle)$
          for any $(x, u, f)\in
          \Gamma(E_0)\oplus\Gamma(E_1)\oplus\Gamma(E_2)=\Gamma(E)$.
    \end{defn}

     \begin{defn}
        Define $\Psi_\iota: \Gamma(E^{\mathfrak{su}})\rightarrow \bigoplus_{k=0}^2\mathrm{Diff}_k(\mathcal{S},
         \mathcal{S})\subset\mathrm{Diff}(\mathcal{S},
         \mathcal{S})$   by  $\Psi_\iota(x, u, f)$ $=(\rho_{\iota(x)}, D_{\iota(u)}, -f\triangle)$
          for any $(x, u, f)\in \Gamma(E^{\mathfrak{su}}_0)\oplus\Gamma(E^{\mathfrak{su}}_1)\oplus\Gamma(E^{\mathfrak{su}}_2)
          =\Gamma(E^{\mathfrak{su}})$. We simply denote $\Psi_\iota$   by
        $\Psi$.
    \end{defn}

 \bigskip

  \begin{thm}\label{ttsu11}
   Let $M$ be a Riemannian manifold with its holonomy group inside $SO(n), U(n)$ or $Sp(n)$.
     Then   $\Omega^\bullet(M)$ admits a
   $\mathfrak{su}_\mathbb{K}(1,1)_{sup}$ action with $\mathbb{K}=\mathbb{R},\mathbb{C}$ or $\mathbb{H}$ respectively.
 \end{thm}
 \begin{remark}
 The $\mathbb{R}$ part of the
    $\mathfrak{su}_\mathbb{K}(1,1)_{sup}$ action consists of $\mathbb{R}\triangle$, where $\triangle=\{D_{\epsilon_1}, D_{\epsilon_2}\}$
     is the  Laplacian   operator $\Delta$ since $D_{\epsilon_1}=d^*$ and $D_{\epsilon_2}=d$ as mentioned in section \ref{op1 sec}.
    Since the $\mathbb{R}$ part is the center of  $\mathfrak{su}_\mathbb{K}(1,1)_{sup}$, the $\mathfrak{su}_\mathbb{K}(1,1)_{sup}$
    action on $\Omega^\bullet(M)$ descends to the cohomology
    $H^*(M)$  by Hodge theory, if $M$ is compact.

 \end{remark}

\begin{proof}[Proof of Theorem \ref{ttsu11}]
     It follows from Example \ref{etrivial} that   $E^{\mathfrak{su}}$ is
     trivial.
     Identify constant sections of $\Gamma(E^{\mathfrak{su}})$ with
     $\mathfrak{su}_\mathbb{K}(1,1)_{sup}$ naturally.
     We need to show  that $\Psi: \mathfrak{su}_\mathbb{K}(1,1)_{sup}\rightarrow
     \mathrm{Diff}(\mathcal{S}, \mathcal{S})$ is an injective Lie superalgebra homomorphism.

     It follows from the
     construction of the operators of order zero that
              $$\Psi([x, y])=\rho_{[x, y]}=[\rho_x, \rho_y]=[\Psi(x), \Psi(y)], \quad\mbox{for any } x, y\in  \mathfrak{su}_\mathbb{K}(1,1).$$
    For any $x\in  \mathfrak{su}_\mathbb{K}(1,1)$ and    $u\in
\mathbb{K}^{1,1}$, $D_ux=0$;   it follows from Proposition \ref{psu}
that
         $$\Psi ([x, u])=D_{x\cdot u}=[\rho_x, D_u]=[\Psi(x), \Psi(u)].$$
     By Proposition \ref{pde}, we have for any $u, v\in
     \mathbb{K}^{1,1}$ that
          $$\Psi([u, v])=\Psi(-2\check q(u, v))=2\check q(u, v)\triangle=\{D_{u}, D_v\}=\{\Psi(u), \Psi(v)\}.$$
     It remains to show that $$[\rho_x+D_u, \triangle]=0,  \quad\mbox{for any } x\in  \mathfrak{su}_\mathbb{K}(1,1)  \mbox{ and any } u\in \mathbb{K}^{1,1}. $$
      In fact,
      $$[\rho_x, \triangle]=\{[\rho_x, D_{\epsilon_1}], D_{\epsilon_2}\}+\{D_{\epsilon_1}, [\rho_x,      D_{\epsilon_2}]\}
                            =2 \check q(x\cdot\epsilon_1, \epsilon_2)\triangle+2\check q({\epsilon_1}, {x\cdot\epsilon_2})\triangle
              =0$$
      Take the decomposition
      $u=u_{1\mathrm{r}}+u_{{1\mathrm{i}}}+u_{2\mathrm{r}}+u_{2\mathrm{i}}$.
      Note that $D_{\epsilon_2}{}^2=0$, it is obvious that
      $[D_{u_{2\mathrm{r}}}, \triangle]=0$.
         We can take $x\in
      \mathfrak{su}_\mathbb{K}(1,1)$ such that $x\cdot
      \epsilon_2=u_{1\mathrm{i}}$ (as we did in the proof of
      Proposition \ref{pde}). Hence,
      $$[D_{u_{1\mathrm{i}}}, \triangle]
              =[[\rho_x, D_{\epsilon_2}], \triangle]=[\rho_x, [D_{\epsilon_2},
              \triangle]]-[D_{\epsilon_2}, [\rho_x, \triangle]]=[\rho_x, 0]-[D_{\epsilon_2},0]=0.$$
       Similarly, we have $ [D_{u_{1\mathrm{r}}+u_{2\mathrm{i}}},
       \triangle]=0$. Hence,
        $$\Psi([x+u, c])=\Psi(0)=0=[\rho_x+D_u, -c\triangle]=[\Psi(x+u), \Psi(c)].$$

Clearly, $\Psi$ is injective; and
        $\Psi(\mathfrak{su}_\mathbb{K}(1,1)_{sup})$, consisting of differential operators, acts on
     $\Omega^\bullet(M)$. Hence,
     $\Omega^\bullet(M)$ admits a
     $\mathfrak{su}_\mathbb{K}(1,1)_{sup}$ action.
\end{proof}

\bigskip

    Note that  the decomposition $\mathbb{K}^{1,1}=\mathbb{R}\epsilon_1\oplus
        \mathrm{Im}\mathbb{K}\epsilon_1\oplus
 \mathbb{R}\epsilon_2\oplus\mathrm{Im}\mathbb{K}\epsilon_2$, induces  a bundle decomposition   $E^{\mathfrak{su}}=E^{\mathfrak{su}}_{1\mathrm{r}}\oplus
        E^{\mathfrak{su}}_{\mathrm{1i}}\oplus E^{\mathfrak{su}}_{\mathrm{2r}}\oplus
        E^{\mathfrak{su}}_{\mathrm{2i}}$ for any normed algebra $\mathbb{K}$.
        Therefore for any   $u\in\Gamma(E^{\mathfrak{su}})$,
        we can write it as
        $u=u_{1\mathrm{r}}+u_{{1\mathrm{i}}}+u_{2\mathrm{r}}+u_{2\mathrm{i}}$.
   Using the same arguments as in the proof of Theorem \ref{ttsu11}, together with   Proposition
   \ref{psu} and Proposition  \ref{pdd}, we have the following
   theorems.

 \begin{thm}\label{tt}
   Let $M$ be an oriented Riemannian manifold. Suppose $M$ is a $\mathbb{K}$-manifold
   with $\mathbb{K}$ being an associative normed algebra.
   Then  $$\sigma\circ\Psi: \Gamma(E^{\mathfrak{su}})\longrightarrow  \mathrm{Symb}(\mathcal{S},
   \mathcal{S})$$
   is a Lie superalgebra monomorphism.
 \end{thm}

  \begin{thm}\label{tb}
   Let $M$ be an oriented Riemannian manifold. Suppose $M$ is a $\mathbb{K}$-manifold
   with $\mathbb{K}$ a normed algebra.
 Then $$\sigma\circ \Psi:
   \iota(\Gamma(E^{\mathfrak{su}}_0))\oplus\big(\iota(\Gamma(E^{\mathfrak{su}}_0))\cdot \iota(\Gamma(E^{\mathfrak{su}}_1))\big)\oplus\iota(\Gamma(E^{\mathfrak{su}}_2))
   \longrightarrow  \mathrm{Symb}(\mathcal{S}, \mathcal{S})$$ is a  Lie superalgebra
   monomorphism.

 \end{thm}

\section{$\mathfrak{su}_{\mathbb{K}'}(2,2)_{sup}$-action for Semi-flat Calabi-Yau and hyperk\"aher manifolds}
Mirror symmetry is a highly nontrivial duality transformation for
 Calabi-Yau manifolds (i.e. special $\mathbb{C}$-manifolds) and
 hyperk\"ahler manifolds (i.e. special
 $\mathbb{H}$-manifolds).
   From the SYZ proposal \cite{SYZ}, mirror
 Calabi-Yau manifolds should admit special Lagrangian fibrations,
 which becomes semi-flat in the large complex structure limit.  Indeed, most Calabi-Yau manifolds which are studied in mirror symmetry are hypersurfaces,
  or more generally, complete intersections, in toric varieties. The conjectural limiting semi-flat structures are expected to come from the toric actions on the ambient toric varieties.
 The hard Lefschetz action should also have a mirror version,
 as it was discussed in \cite{leung2}. We conjecture that this
 mirror hard Lefschetz action should be closely related to the Schmid
 $SL_2$-orbit theorem for the large complex structure
 degeneration.

  Putting both $\mathfrak{su}(1,1)$ actions together, we have a
  $\mathfrak{su}(1,1)\oplus
    \mathfrak{su}(1,1)=\mathfrak{so}(2,2)$ action on differential
    forms on semi-flat Calabi-Yau manifolds. We are going to explain
    this enlarged (super) hard Lefschetz action below.
    In this article, we use the following definition of
    semi-flatness.

    \begin{defn}\label{semi}
      A   $\mathbb{K}$-manifold is called semi-flat if its
      holonomy group can be reduced from $G_\mathbb{K}$ to
      $G_{\mathbb{K}'}^\circ$, the connected component of  $G_{\mathbb{K}'}$. Here $\mathbb{K}'$ means $\mathbb{R}$ or
      $\mathbb{C}$ when $\mathbb{K}$ equals $\mathbb{C}$ or
      $\mathbb{H}$ respectively.
    \end{defn}

    For example, given any (open) Calabi-Yau manifold $M$ of real dimension $2n$ with a free Hamiltonian $T^n$-action preserving the Calabi-Yau structure. The
    K\"ahler potential $\varphi$ of $M$ can be descended to a function on the quotient manifold $B=M/T^n$ and induces a Riemannian metric $g_B$ of Hessian type on $B$, namely
     $g_B=\nabla^2B$, where the Hessian $\nabla^2B$ is computed  with respect to the canonical affine structure on $B$ induced from the Lagrangian fibration structure on $M$. Furthermore, the holonomy group of $M$ and
     $B$ are the same. Thus the holonomy group of $M$ is inside $SU(n)\cap GL(n,\mathbb{R})=SO(n)$, and therefore $M$ is a semi-flat Calabi-Yau manifold as in Definition \ref{semi}.
     There are similar constructions (namely $T^n$-invariant hyperk\"ahler manifolds in \cite{leung2}) for a class of semi-flat hyperk\"ahler manifolds. Topologically, they are always
    products of domains in $\mathbb{R}^n$ with tori. Despite such severe restrictions on their geometry, they are expected to arise naturally in the large complex structure limit
    and play an important role in mirror symmetry as indicated in the SYZ proposal.

       The tangent bundle $\mathbb{K}^n\rightarrow
       T\rightarrow M$ of  a semi-flat $\mathbb{K}$-manifold $M$ is the
       complexification of another bundle $(\mathbb{K'})^n\rightarrow
       T'\rightarrow M$.

       Recall when $V\cong \mathbb{K}^n$ then $V\oplus V^*$ with the
       canonical quadratic form $Q$ identifies it with
       $\mathbb{K}^{n}\otimes_\mathbb{K}\mathbb{K}^{1,1}$. Thus $
       \mathfrak{su}_\mathbb{K}(1,1)$  acts on $V\oplus V^*$ and its
       spinor representation $S=\bigwedge^\bullet V^*$. Now $V\cong
       V'\otimes_\mathbb{R}\mathbb{C}$ with  $V'\cong
       (\mathbb{K'})^n$. By same reasonings, we have
          $$V\oplus V^*\cong (\mathbb{K'})^{n}\otimes_{\mathbb{K}'}(\mathbb{K}')^{2,2}$$
     Thus we obtain a $ \mathfrak{su}_{\mathbb{K'}}(2,2)$ action on $(V\oplus V^*, Q)$,
       and therefore also on its spinor representation $S=\bigwedge^\bullet V^*$.
       Furthermore this action commutes with the natural
       $\mathfrak{u}_{\mathbb{K}'}(n)$ action. Therefore, we obtain
       a $ \mathfrak{su}_{\mathbb{K'}}(2,2)$ action on the space of
       differential forms on a semi-flat $\mathbb{K}$-manifold $M$.
       One can check directly that for semi-flat Calabi-Yau
       manifolds, this
       $\mathfrak{so}(2,2)=\mathfrak{sl}(2,\mathbb{R})\oplus
    \mathfrak{sl}(2, {\mathbb{R}})$  action corresponds to the hard
    Lefschetz action and its mirror action as defined in
    \cite{leung2} (see also \cite{caozhou}).

             To see these Lie algebras concretely, we note that
             $\mathfrak{su}_{\mathbb{K}'}(2,2)\cong\mathfrak{so}(\dim
             \mathbb{K},2)$.
         \begin{align*}
               &\mathfrak{su}_\mathbb{C}(1,1)=\mathfrak{so}(2,1)\subset\mathfrak{su}_{\mathbb{R}}(2,2)=\mathfrak{so}(2, 2)\\
               &\mathfrak{su}_\mathbb{H}(1,1)=\mathfrak{so}(4,1)\subset\mathfrak{su}_{\mathbb{C}}(2,2)=\mathfrak{so}(4, 2).
         \end{align*}

    Clearly
    $\mathfrak{su}_{\mathbb{K}'}(2,2)_{sup}=\mathfrak{su}_{\mathbb{K}'}(2,2)\oplus
    (\mathbb{K}')^{2,2}\oplus \mathbb{R}$ is naturally a Lie
    superalgebra, which includes
    $\mathfrak{su}_\mathbb{K}(1,1)_{sup}$ as a super Lie subalgebra.
   Thus, the real vector space $(\mathbb{K}')^{2,2}\oplus \mathbb{R}=\mathbb{K}^{1,1}\oplus \mathbb{R}$
   acts on $\Omega^\bullet(M)$ via
      differential operators of order one and two. Together with the  $\mathfrak{su}_{\mathbb{K}'}(2,2)$ action,
      which extends the  $\mathfrak{su}_{\mathbb{K}}(1,1)$ action, it gives a
      Lie superalgebra $\mathfrak{su}_{\mathbb{K}'}(2,2)_{sup}$ action on  $\Omega^\bullet(M)$.

       In conclusion, we have obtained the following result for semi-flat Calabi-Yau and hyperk\"ahler manifolds.
     \begin{thm}
       Suppose that $M$ is a semi-flat $\mathbb{K}$-manifold with
       $\mathbb{K}$ being $\mathbb{C}$ or $\mathbb{H}$, then there
       is a natural $\mathfrak{su}_{\mathbb{K}'}(2,2)_{sup}$ action,
       extending the super hard Lefschetz $\mathfrak{su}_{\mathbb{K}}(1,1)_{sup}$ action,
       on the space of differential forms on $M$ via differential
       operators of order at most two.
     \end{thm}

\section{Appendix}

\subsection{$\mathfrak{su}_\mathbb{K}(1, 1)\cong \mathfrak{so}(\dim\mathbb{K}, 1)$}

  There is a natural isomorphism $\tau_*: \mathfrak{sl}(2, \mathbb{K})\overset{\cong}{\longrightarrow}
     \mathfrak{so}(\mathbb{R}^{1,1}\oplus\mathbb{K})$. One can refer to \cite{baez}
     for the  geometric meaning of the isomorphism.       Furthermore,
     we have
      $$\tau_*|_{\mathfrak{su}_\mathbb{K}(1,1)}: \mathfrak{su}_\mathbb{K}(1,1)\overset{\cong}{\longrightarrow}\mathfrak{so}(\mathbb{R}^{1,1}\oplus\mathrm{Im}\mathbb{K})
     \cong \mathfrak{so}(\dim_\mathbb{R}\mathbb{K}, 1).$$
  We will write down it more explicitly for the case $\mathbb{K}$ is associative.
     Identify $ \mathbb{R}^{1,1}\oplus\mathbb{K}$ with
   $\mathfrak{h}_2(\mathbb{K})$, the hermitian $2\times 2$ matrices with entries in $\mathbb{K}$, via the map  $( \alpha, \beta , x)\mapsto
             \left(\begin{array}{cc}   \alpha+\beta &x \\
                           \bar x &\alpha-\beta \\\end{array}\right)$,
     where $ \alpha, \beta\in \mathbb{R},\,x\in  \mathbb{K}.$
  Then there is a
   double cover  given by
    $  \tau: SL(2,\mathbb{K})\longrightarrow
              SO^+(\mathbb{R}^{1,1}\oplus\mathbb{K});\;
             A \mapsto \tau_A $,
                 where
         \begin{align*} \tau_A:& \quad  \mathbb{R}^{1,1}\oplus\mathbb{K} \longrightarrow
\mathbb{R}^{1,1}\oplus\mathbb{K};\\
                     &\Bigg(\begin{array}{cc}
                 \alpha+\beta &x \\
               \bar x &\alpha-\beta \\\end{array}\Bigg) \mapsto A \Bigg(\begin{array}{cc}
                 \alpha+\beta &x \\
                  \bar x &\alpha-\beta \\\end{array}\Bigg)A^\star.
          \end{align*}
  Therefore, it induces an isomorphism $\tau_*$ of Lie algebras.

  For the associative  normed algebra $\mathbb{K}$, the natural inclusion
   $\mathbb{R}^{1,1}\oplus
     \mathrm{Im}\mathbb{K}\hookrightarrow\mathbb{R}^{1,1}\oplus
     \mathrm{Im}\mathbb{H}\hookrightarrow  \mathbb{R}^{1,1}\oplus
     \mathbb{H}$, induces an embedding of $\mathfrak{so}(\mathbb{R}^{1,1}\oplus\mathrm{Im}\mathbb{K})$ into $\mathfrak{so}(1, 1+\dim_\mathbb{R}\!\mathbb{H})$
   naturally. Therefore we only write down  $\tau_*( \mathfrak{su}_\mathbb{H}(1,1))$
     explicitly.
    Let $E_{ij}$ be the matrix with $1$ in the $(i, j)$th entry and 0 elsewhere.
    Take the basis of
         $\mathfrak{su}_\mathbb{H}(1,1)$ as
       in  section \ref{opall}, then we have
       \begin{align*}
         \tau_*(L_s)&=
         E_{1(3+s)}+E_{(3+s)1}-E_{2(3+s)}+E_{(3+s)2},\,\, s=1,2,3;\\
          \tau_*(\Lambda_s)&=E_{1(3+s)}+E_{(3+s)1}+E_{2(3+s)}-E_{(3+s)2} ,\,\, s=1,2,3;\\
         \tau_*(K_1)&= 2(E_{65}-E_{56}),\qquad\quad \tau_*(K_2)=
         2(E_{46}-E_{64}) ;\\
         \tau_*(K_3)&= 2(E_{54}-E_{45}),\qquad\quad\,\,\, \tau_*(H)=  2(E_{12}+E_{21}).
       \end{align*}

\subsection{Proof of Proposition \ref{pdd}}
  We use the notation of $k$-symbol $\sigma_k$ as in \cite{wells} for
       differential operators.

       For any $p\in M$, let $(y_1, \cdots, y_m)$ be a normal coordinate system around  $p$.
     Then for any $u\in \Gamma(E^{\mathfrak{su}}_1)$, $D_u=\sum_{j=1}^m\iota(u)(dy^j)\cdot\nabla_{\!\!{\partial \over \partial y_j}}$.

      For any $\xi\in  T^*_pM$ and any $\varphi\in \bigwedge^\bullet T^*_pM$, take
       $g\in \Omega^0(M)$ and $s\in \Gamma(\mathcal{S})$ such that
       $dg(p)=\xi$ (i.e. $\sum_j{\partial g \over \partial y_j}(p)dy^j=\xi$) and $s(p)=\varphi$, then we have $\sigma_k(D_u)(p,
           \xi)\varphi=0$ for any $k\geq 2$, and

   $$\sigma_1(D_u)(p,
           \xi)\varphi=\Big(D_u\big({i\over 1!}(g-g(p))s\big)\Big)(p)
           =\sum_{j=1}^m\iota(u)(dy^j)\cdot{\partial g \over \partial y_j}(p)\varphi.  $$
    In particular, $\sigma_1(D_{\epsilon_2})(p,
           \xi)\varphi= dy^j \cdot{\partial g \over \partial y_j}(p)\varphi$ and $\sigma_1(D_{\epsilon_1})(p,
           \xi)\varphi= {\partial \over \partial y_j} \cdot{\partial g \over \partial y_j}(p)\varphi$.
   Hence,  $$\sigma_2(\triangle)(p,
           \xi)\varphi=\sum_{j, k} {\partial g \over \partial y_j}(p){\partial g \over \partial y_k}(p)(dy^j \cdot
             {\partial \over \partial y_k} +{\partial \over \partial y_k} \cdot dy^j)\cdot
             \varphi= -\sum_{j=1}^m \big({\partial g \over \partial y_j}(p)\big)^2 \varphi.$$
      On the other hand,
         \begin{align*}&\sigma_2(\{D_u, D_v\})(p, \xi)\varphi\\
                 =&\sum_{j, k}\iota(u)(dy^j)\cdot{\partial g \over \partial y_j}(p)\iota(v)(dy^k)\cdot{\partial g \over \partial y_k}(p)\varphi
                    \!\!+\!\!\sum_{j, k}\iota(v)(dy^j)\cdot{\partial g \over \partial y_j}(p)\iota(u)(dy^k)\cdot{\partial g \over \partial y_k}(p)\varphi
                         \\
                =&\sum_{j, k}{\partial g \over \partial y_j}(p){\partial g \over \partial
                y_k}(p)\big(\iota(u)(dy^j)\cdot\iota(v)(dy^k)
                    + \iota(v)(dy^k)\cdot \iota(u)(dy^j)\big)\cdot\varphi
                         \\
                 =&\sum_{j, k}-{\partial g \over \partial y_j}(p){\partial g \over \partial
                y_k}(p) \cdot 2Q\big(\iota(u)(dy^j), \iota(v)(dy^k) \big)\cdot\varphi
                         \\
                  =&\sum_{j, k}-2{\partial g \over \partial y_j}(p){\partial g \over \partial
                y_k}(p)\cdot \check q(u, v)\delta_{jk}\cdot\varphi
                         \\
                    =&-2\check q(u, v)\sum_{j=1}^m\big({\partial g \over \partial
                    y_j}(p)\big)^2\cdot\varphi.
        \end{align*}

   Hence, $\sigma_2(\{D_u, D_v\})(p, \xi)\varphi=\sigma_2(2\check q(u, v)\triangle)(p,
           \xi)\varphi$.

Therefore, $\sigma_2(\{D_u, D_v\}-2\check q(u, v)\triangle)=0.$

 Since $D_u$ and $D_v$ are of   order one, $\{D_u, D_v\}-2\check q(u,
   v)\triangle$ is of order   at most two. Therefore,
   $\{D_u, D_v\}-2\check q(u, v)\triangle$ is a first order
   operator.

\subsection{\label{opall} Identifying $\mathfrak{su}_\mathbb{K}(1,1)_{sup}$ with the usual hard Lefschetz actions}
    We can reinterpret those operators in the $\mathfrak{su}_\mathbb{K}(1,1)_{sup}$ action in Theorem \ref{ttsu11} as
    follows.

      \bigskip

    \noindent Case $\mathbb{K}=\mathbb{R}$:

          In this case, $G_\mathbb{K}^\circ=SO(n)$ and $M$ is an oriented
        Riemannian manifold. Furthermore we have
        $\mathfrak{su}_\mathbb{R}(1,1)_{sup}=\mathbb{R}h\oplus \mathbb{R}^{1,1}\oplus \mathbb{R}$, where
      $h=\Big(\begin{array}{cc}1&0\\0&-1\end{array}\Big)$, so that
    \vspace{-0.2cm}

          $$\rho_h|_{\Omega^p(M)}=({m\over
          2}-p)\mathrm{Id},\quad
         D_{\epsilon_2}=d,\quad D_{\epsilon_1}=d^*,\quad \Psi(1)=-\Delta.
          $$


    \noindent Case $\mathbb{K}=\mathbb{C}$:

          In this case, $G_\mathbb{K}^\circ=G_\mathbb{K}=U(n)$ and $M$ is a  K\"ahler
          manifold with K\"ahler form $\omega$. Moreover, we have
            $\mathfrak{su}_\mathbb{C}(1,1)_{sup}=\mathfrak{su}_\mathbb{C}(1,1)\oplus \mathbb{C}^{1,1}\oplus \mathbb{R}$
            with
   $\mathfrak{su}_\mathbb{C}(1,1)=\mbox{Span}_\mathbb{R}\{L, \Lambda, H\}$,  where
   $$ L={\scriptstyle \left(\begin{array}{cc} 0 & 0  \\
                             -\sqrt{-1} & 0
                \\\end{array}\right)}\,,\, \Lambda={\scriptstyle\left(\begin{array}{cc} 0 & \sqrt{-1}  \\
                             0  & 0
                \\\end{array}\right)}\,,\,H={\scriptstyle\left(\begin{array}{cc} -1 & 0  \\
                              0   & 1
                \\\end{array}\right)}.$$

  And we have $\rho_L=\omega\wedge$, $\rho_\Lambda=\rho_L^*$ and $\rho_H=[\rho_L, \rho_\Lambda]$, which
   are exactly  those defining the hard
  Lefschetz action on K\"ahler manifolds \cite{gri} (see also \cite{li} for more
details). On complex valued differential forms, we have
$$ \begin{array}{lcl} D_{\epsilon_2}=d=\partial+\bar\partial, &  & D_{\sqrt{-1}\epsilon_2}=\sqrt{-1}(\bar\partial-\partial),\\
                         D_{\epsilon_1}=d^*=\partial^*+\bar\partial^*,&  & D_{\sqrt{-1}\epsilon_1}=\sqrt{-1}(\bar\partial^*-\partial^*),\\
                         \{D_{\epsilon_1}, D_{\epsilon_2}\}=\{D_{\sqrt{-1}\,\epsilon_1}, D_{\sqrt{-1}\,\epsilon_2}\}=\Delta=-\Psi(1).& &\\
                    \end{array}
           $$


    \noindent Case $\mathbb{K}=\mathbb{H}$:

          In this case, $G_\mathbb{K}^\circ=G_\mathbb{K}=Sp(n)$ and $M$ is a  hyperk\"ahler
          manifold. Furthermore,
  $\mathfrak{su}_\mathbb{H}(1,1)=\Big\{\Bigg(\begin{array}{cc} \beta_1
& \beta_2\\\beta_3 &-\bar\beta_1
 \end{array}\Bigg) ~\Big|~ \beta_1\in\mathbb{H},\,\, \beta_2,\beta_3\in
 \mbox{Im}\mathbb{H} \Big\}$ is ten dimensional, and is spanned  by $\{L_s, \Lambda_s, K_s, H~|~ s=1,2,3\}$, where
   $$ L_s={\scriptstyle \left(\begin{array}{cc} 0 & 0  \\
                              -J_s& 0
                \\\end{array}\right)}\,,\, \Lambda_s={\scriptstyle\left(\begin{array}{cc} 0 & J_s  \\
                             0  & 0
                \\\end{array}\right)},K_s={\scriptstyle \left(\begin{array}{cc} J_s & 0  \\
                              0& J_s
                \\\end{array}\right)}\,, H={\scriptstyle \left(\begin{array}{cc} -1 & 0  \\
                             0  & 1
                \\\end{array}\right)}$$
  and $J_1^2=J_2^2=J_3^2=J_1J_2J_3=-1$.

  $\rho_{L_s}$ is exactly the
same as the operator ``$\omega_s\wedge$" where $\omega_s$ is the
K\"ahler form with respect to the complex structure $J_s$, and
$\rho_{\Lambda_s}$ is  the adjoint operator of $\rho_{L_s}$ for each
$s$. Moreover, for each $s\in\{1, 2, 3\}$,
$$ \begin{array}{lcl} D_{\epsilon_2}=d=\partial_s+\bar\partial_s, &  & D_{J_s\epsilon_2}=\sqrt{-1}(\bar\partial_s-\partial_s),\\
                         D_{\epsilon_1}=d^*=\partial_s^*+\bar\partial_s^*,&  & D_{J_s\epsilon_1}=\sqrt{-1}(\bar\partial_s^*-\partial_s^*),\\
                         \{D_{\epsilon_1}, D_{\epsilon_2}\}=\{D_{J_s\epsilon_1}, D_{J_s\epsilon_2}\}=\Delta=-\Psi(1).& &\\
                    \end{array}
           $$

\section*{Acknowledgements}
 The authors   thank the referee
for useful suggestions  on an earlier version of this paper.
     The first author is supported in part by      RGC Grant CUHK2160256 from the Hong Kong Government.

\bibliographystyle{mrl}

}
\end{document}